\documentclass[12pt]{article}

\usepackage{amssymb, amsmath, amsthm}
\usepackage{fullpage}
\usepackage{tikz}

\DeclareMathOperator{\ex}{ex}

\newtheorem{theorem}{Theorem}

\newtheorem{lemma}[theorem]{Lemma}

\newtheorem*{claim}{Claim}

\newtheorem*{remark}{Remark}

\numberwithin{equation}{section}

\begin{document}

\title{Rainbow cycles vs. rainbow paths}
\author{Anastasia Halfpap\thanks{Department of Mathematical Sciences,
University of Montana, Missoula, Montana 59812, USA.}
\and
Cory Palmer\footnotemark[1]
}
\maketitle

\begin{abstract}
    An edge-colored graph $F$ is {\it rainbow} if each edge of $F$ has a unique color. The {\it rainbow Tur\'an number} $\ex^*(n,F)$ of a graph $F$ is the maximum possible number of edges in a properly edge-colored $n$-vertex graph with no rainbow copy of $F$. The study of rainbow Tur\'an numbers was introduced by Keevash, Mubayi, Sudakov, and Verstra\"ete. Johnson and Rombach introduced the following rainbow-version of generalized Tur\'an problems: for fixed graphs $H$ and $F$, let $\ex^*(n,H,F)$ denote the maximum number of rainbow copies of $H$ in an $n$-vertex properly edge-colored graph with no rainbow copy of $F$.
    
    In this paper we investigate the case $\ex^*(n,C_\ell,P_\ell)$ and give a general upper bound as well as exact results for $\ell = 3,4,5$. Along the way we establish a new best upper bound on $\ex^*(n,P_5)$. Our main motivation comes from an attempt to improve bounds on $\ex^*(n,P_\ell)$, which has been the subject of several recent manuscripts.
\end{abstract}

\section{Introduction}
We say that an edge-colored graph is \textit{rainbow} if no two edges receive the same color. We use the term {\it rainbow}-$F$ to refer to a rainbow copy of the graph $F$.
A properly edge-colored graph is {\it rainbow-$F$-free} if it contains no rainbow copy of $F$ as a subgraph.
The \textit{rainbow Tur$\acute{a}$n number} of a fixed graph $F$ is the maximum possible number of edges in a properly edge-colored $n$-vertex rainbow-$F$-free graph $G$. We denote this maximum by $\ex^*(n,F)$. The study of rainbow Tur\'an numbers was introduced by  Keevash, Mubayi, Sudakov, and Verstra\"ete in \cite{kmsv}.

Observe that $\ex(n,F) \leq \ex^*(n,F)$, since any $F$-free graph clearly contains no rainbow-$F$. In fact, it was proved in \cite{kmsv} that for any $F$,
\[
\ex(n,F) \leq \ex^*(n,F) \leq \ex(n,F) + o(n^2).
\]
However, for bipartite $F$, $\ex(n,F)$ and $\ex^*(n,F)$ are not asymptotic in general.
For example, in \cite{kmsv} it was shown that asymptotically $\ex^*(n,C_6)$ is a constant factor larger than $\ex(n,C_6)$. Another interesting example concerns acyclic graphs. The maximum number of edges in an $n$-vertex graph containing no cycle is $n-1$. On the other hand, the maximum number of edges in a properly edge-colored $n$-vertex graph with no rainbow cycle is at least $n \log n$ (see \cite{kmsv}). Moreover, Das, Lee and Sudakov \cite{DLS13} showed that this maximum is at most $n^{1+\epsilon}$ for any $\epsilon >0$ and $n$ large enough.

We denote by $P_\ell$ the path on $\ell$ edges, i.e., $\ell+1$ vertices. The behavior of $\ex^*(n,P_\ell)$ is known only for $\ell \leq 4$. For $\ell = 1$ and $\ell = 2$, the result $\ex^*(n, P_\ell) = \ex(n,P_\ell)$ is trivial, since any properly colored $P_1$ or $P_2$ is rainbow. For $\ell = 3$ and $\ell = 4$, Johnston, Palmer and Sarkar \cite{jps} showed:

\begin{theorem}[Johnston-Palmer-Sarkar \cite{jps}]\label{path-edge-bounds}
If $n$ is divisible by $4$, then 
\[\ex^*(n,P_3) = \frac{3}{2}n.
\]
If $n$ is divisible by $8$, then
\[
\ex^*(n,P_4) = 2n.
\]
\end{theorem}

The best-known general lower bound and upper bound are due to Johnston and Rombach \cite{jr} and Ergemlidze, Gy\H{o}ri and Methuku \cite{EGyM19}, respectively:

\begin{theorem}[Johnston-Rombach \cite{jr}; Ergemlidze-Gy\H{o}ri-Methuku \cite{EGyM19}]\label{edge-bound}
For $\ell \geq 3$,
\[
\frac{\ell}{2}n\leq \ex^*(n,P_\ell)\leq \left(\frac{9\ell+5}{7}\right)n.
\]
\end{theorem}

Johnston and Rombach also considered a rainbow-version of generalized Tur\'an problems popularized by Alon and Shikhelman \cite{AS}. For fixed graphs $H$ and $F$, let $\ex^*(n,H,F)$ denote the maximum possible number of rainbow copies of $H$ in an $n$-vertex properly edge-colored graph with no rainbow-$F$. In this paper we will be primarily concerned with determining the value of $\ex^*(n,P_\ell,C_\ell)$. Our motivation comes from the investigation of $\ex^*(n,P_\ell)$, but the study of $\ex^*(n,H,F)$ is a natural analogue of generalized Tur\'an problems and rainbow Tur\'an problems.

Let us mention that other generalizations have been investigated. Gerbner, M\'esz\'aros, Methuku and Palmer \cite{GMMP} investigated the function $\ex(n,H,\textrm{rainbow-}F)$ which is the maximum number of copies of $H$ in a properly edge-colored $n$-vertex graph with no rainbow-$F$. In most cases they considered the case when $H=F$. The case $H=F= K_k$ was considered recently by Gowers and Janzer \cite{GJ}.

The construction achieving the lower bound for $\ex^*(n,P_4)$ contains many rainbow walks of length $4$, but as there is no rainbow-$P_4$, each of these walks must be a cycle. In fact, this construction has the maximum number of rainbow-$C_4$ copies without a rainbow-$P_4$. A better understanding of $\ex^*(n,C_\ell,P_\ell)$ may help improve the bounds on $\ex^*(n,P_\ell)$.

The main results in this paper are summarized in the following theorem.

\begin{theorem}\label{main} For $\ell \geq 3$,
\[
\frac{(\ell-1)!}{2}n \leq \ex^*(n,C_\ell,P_\ell) \leq (2\ell-3)^{\ell-2} \cdot \ex^*(n,P_\ell) \leq c(\ell) n
\]
for some constant $c(\ell)$ depending on $\ell$.
Moreover, for $\ell =3,4,5$ we have
\[
\ex^*(n,C_\ell,P_\ell) = \frac{(\ell-1)!}{2}n
\]
when $n$ is divisible by $2^{\ell-1}$.
\end{theorem}

%In particular, taking copies of $D_{2^{\ell}}^*$ achieves approximately $\ex^*(n,C_\ell,P_\ell)$ for $\ell = 3,4,5$.

In Section~\ref{general-section} we give simple general bounds on $\ex^*(n,C_\ell,P_\ell)$ which gives the first part of Theorem~\ref{main}. In Section~\ref{small-bounds} we give  matching upper bounds when $\ell=3,4,5$ and $n$ is divisible by $2^{\ell-1}$. Note that this immediately implies tight asymptotic bounds for all $n$.

\section{General bounds}\label{general-section}

We begin with the construction from \cite{jr} giving a lower bound on $\ex^*(n,P_\ell)$.
Let $Q_{\ell-1}$ be the $\ell-1$ dimensional cube, i.e., the graph whose vertex set is the set of $01$-strings of length $\ell-1$ and two vertices are joined by an edge if and only if their Hamming distance is exactly $1$.

Now let us color the edges of $Q_{\ell-1}$ by the position in which their corresponding strings differ. For each vertex $x$ of 
$Q_{\ell-1}$, let $\overline{x}$ be the {\it antipode} of $x$. That is, $\overline{x}$ is the unique vertex of Hamming distance $\ell-1$ from $x$ (i.e. all bits of $x$ are swapped). Now add all edges $x\overline{x}$ to this graph and color these edges with a new color $\ell$. Call these edges {\it diagonal} edges and denote the resulting edge-colored graph $D_{2^{\ell-1}}^*$.

It is easy to see that the edge-coloring above is proper. It was shown in \cite{jr} that $D_{2^{\ell-1}}^*$ contains no rainbow-$P_\ell$. Let us give another argument here for completeness. Suppose that $D_{2^{\ell-1}}^*$ contains a rainbow path $P$ of length $\ell$. The path $P$ must include an edge $x\overline{x}$ of color $\ell$. Removing the edge $x\overline{x}$ from $P$ leaves two subpaths $P'$ and $P''$ (allowing for $P''$ to be the empty path when $P$ ends with edge $x\overline{x}$). The subpath $P'$ corresponds to bit changes to $x$ and, as $P$ is rainbow, $P''$ corresponds to the complement of these bit changes starting with $\overline{x}$. Therefore, $P'$ and $P''$ share an end-vertex $y$ (allowing for $y = \overline{x}$ when $P''$ is empty), i.e., $P$ is a cycle, a contradiction.

\begin{theorem}\label{lower-bound}
For $\ell \geq 3$, we have $\frac{(\ell-1)!}{2}n \leq \ex^*(n, C_\ell, P_\ell)$ when $n$ is divisible by $2^{\ell-1}$. %Thus, $\frac{(\ell-1)!}{2}n + O(1) \leq \ex^*(n, C_\ell, P_\ell)$ in general.
\end{theorem}

\begin{proof}
Let $G$ be a graph of $n/2^{\ell-1}$ vertex-disjoint copies of $D_{2^{\ell-1}}^*$.
As each copy of $D_{2^{\ell}}^*$ has exactly $\ell$ edge colors, any rainbow-$C_\ell$ must contain an edge of color $\ell$. Recall that edges of color $\ell$ are the diagonal edges.

Fix a diagonal edge $x\overline{x}$ and count the number of rainbow-$C_\ell$ copies containing $x\overline{x}$. This is precisely the number of length-$(\ell-1)$ rainbow paths between $x$ and $\overline{x}$ colored from $\{1,2,\dots,\ell-1\}$. Each such rainbow-$P_{\ell-1}$ is obtained by a sequence of $\ell-1$ bit changes. There are $(\ell-1)!$ distinct sequences, each of which produces a distinct rainbow path between $x$ and $\overline{x}$, so $x\overline{x}$ is included in $(\ell-1)!$ rainbow-$C_\ell$ copies. There are $2^{\ell-2}$ diagonal edges in each $D_{2^{\ell-1}}^*$ (one for each antipodal pair $x$, $\overline{x}$), and so a copy of $D_{2^{\ell-1}}^*$ contains a total of 
\[
(\ell-1)!\cdot 2^{\ell-2}
\] 
rainbow-$C_\ell$ copies. Thus the total number of rainbow-$C_\ell$ copies in $G$ is
\[
(\ell-1)!\cdot 2^{\ell-2}\cdot {\frac{n}{2^{\ell-1}}}   = \frac{(\ell-1)!}{2}n.
\]
\end{proof}

We need the following simple lemma which will also be useful in Section~\ref{small-bounds}.

\begin{lemma}\label{k-edge-colors}
Fix integers $k\geq \ell\geq 1$.
If $G$ is a properly $k$-edge-colored graph and $xy$ is an edge of $G$, then $xy$ is contained in at most $\frac{(k-1)!}{(k-\ell)!}$ rainbow-$C_\ell$ copies. In particular, if $k=\ell$, then $xy$ is contained in at most $(\ell-1)!$ rainbow-$C_\ell$ copies.
\end{lemma}

\begin{proof}
Note that the rainbow-$C_\ell$ copies containing an edge $xy$ correspond to the rainbow paths of length $\ell-1$ with endpoints $x$ and $y$ which do not use the color on $xy$. For each rainbow path $P = xv_1v_2\cdots v_{\ell-2}y$, associate to $P$ the ordered list of edge colors $(c(xv_1),\dots,c(v_{\ell-2}y))$. There are $\frac{(k-1)!}{(k-\ell)!}$ possible distinct lists, so we are done as long as no two distinct paths between $x$ and $y$ are associated to the same list. Suppose to the contrary that $P_1 = xv_1v_2\cdots v_{\ell-1}y$ and $P_2 = xw_1w_2\cdots w_{\ell-1}y$ are distinct rainbow paths with $(c(xv_1), c(v_1v_2),\dots ,c(v_{\ell-2}y)) = (c(xw_1), c(w_1w_2),\dots ,c(w_{\ell-2}y))$. Since $P_1$ and $P_2$ are distinct, there is a smallest index $i$ such that $v_i \neq w_i$; clearly, $i \geq 1$. But (making, if necessary, the identifications $x = v_0 = w_0$), we have $c(v_{i-1}v_i) = c(w_{i-1}w_i)$. By the choice of $i$, $v_{i-1} = w_{i-1}$. This is a contradiction to the hypothesis that $G$ is properly $k$-edge-colored, since now $v_{i-1}v_i$ and $v_{i-1}w_i$ are distinct edges incident to $v_{i-1}$ which receive the same color.
\end{proof}

In a proper $\ell$-edge-coloring each rainbow-$C_\ell$ contains an edge of color $1$. In an $n$-vertex graph there are at most $\frac{n}{2}$ edges of color $1$. Therefore, Lemma~\ref{k-edge-colors} implies that there are at most ${(\ell-1)!} \cdot \frac{n}{2}$ rainbow-$C_\ell$ copies. However, it is not clear that using only $\ell$ edge colors is optimal. A proof of this fact would determine $\ex^*(n,C_\ell,P_\ell)$, but this appears to be difficult.

We now give an upper bound on $\ex^*(n,C_\ell,P_\ell)$ for general $\ell$. The heart of the argument is the following simple lemma:

\begin{lemma}\label{degree-lemma}
Let $G$ be a properly edge-colored graph with no rainbow-$P_\ell$. If $v_1v_2 \cdots v_\ell v_1$ is a rainbow-$C_\ell$ in $G$, then $d(v_i) \leq 2\ell -3$ for $1 \leq i \leq \ell$.
\end{lemma}

\begin{proof}
Consider a vertex $v_i$ on a rainbow cycle $C=v_1v_2 \cdots v_\ell v_1$. Each edge $v_ix$ where $x$ is not on $C$ must be colored with a color used on an edge of $C$ (that is not incident to $v_i$) as otherwise we can construct a rainbow-$P_\ell$. Thus, there are at most $\ell-2$ such edges $v_ix$. Moreover, $v_i$ is adjacent to at most $\ell-1$ other vertices on $C$. Therefore, $d(v_i) \leq 2\ell-3$.
\end{proof}

In general a rainbow-$C_\ell$ cannot have many vertices of degree $2\ell-3$. In Section~\ref{small-bounds} we will make a deeper analysis of vertex degrees in the case when $\ell=3,4,5$ to prove a stronger result.

\begin{theorem}\label{general-upper}
For $\ell \geq 3$,
\[
\ex^*(n,C_\ell,P_\ell) \leq (2\ell-3)^{\ell-2} \cdot \ex^*(n,P_\ell) \leq c(\ell) n
\]
for some constant $c(\ell)$ depending on $\ell$.
\end{theorem}

\begin{proof}
Let $G$ be a properly edge-colored graph on $n$ vertices that does not contain a rainbow-$P_\ell$. Fix an edge $v_1v_2$ and bound the number of rainbow-$C_\ell$ copies containing $v_1v_2$. We may assume that $v_1v_2$ is contained in at least one rainbow-$C_\ell$.

Note that the number of rainbow-$C_\ell$ copies containing $v_1v_2$ is bounded above by the number of ways in which we can pick $\ell-1$ more edges to form a cycle $v_1v_2 \cdots v_\ell v_1$. By Lemma~\ref{degree-lemma}, $d(v) \leq 2\ell - 3$ for each $v$ in a cycle with $v_1v_2$, so there are at most $2\ell-3$
ways in which to chose each vertex. Therefore, the number of rainbow-$C_\ell$ copies containing $v_1v_2$ is bounded above by $(2\ell - 3)^{\ell-2}$. Therefore, the number of rainbow-$C_\ell$ copies is at most
\[
(2\ell - 3)^{\ell-2} \cdot \ex^*(n,P_\ell)
\]
which is linear in $n$ by Theorem~\ref{edge-bound}.
\end{proof}

\section{Asymptotic bounds}\label{small-bounds}
For small values of $\ell$, we can determine $\ex^*(n,C_\ell,P_\ell)$ exactly when $n$ is divisible by $2^{\ell-1}$. For the remaining values of $n$ this gives tight asymptotic bounds.

\begin{theorem}\label{length-3}
If $n$ is divisible by $4$, then
$\ex^*(n,C_3,P_3) = n$. %In general, $\ex^*(n,C_3,P_3) = n + O(1)$
\end{theorem}

\begin{proof}
Theorem~\ref{lower-bound} gives $n = \frac{(3-1)!}{2}n \leq \ex^*(n, C_3,P_3)$.
For the upper bound, let $G$ be an $n$-vertex graph with a proper edge-coloring with no rainbow-$P_3$. Note that every $C_3$ is rainbow, so it suffices to count the number of triangles. Thus, let us count the number of triangles containing a fixed edge $xy$. We may assume that $xy$ is contained in at least one triangle, say $xyzx$. A triangle containing $xy$ which is distinct from $xyzx$ is of the form $xyvx$, and so if $xy$ is in two triangles, then $d(x) \geq 3$ and $d(y) \geq 3$ (since $xv$, $yv$, $xz$, $yz$, and $xy$ are all edges in $G$). Observe that to avoid a rainbow-$P_3$, $x,y$ and $z$ all must have degree at most $3$. Thus, if $xy$ is in two triangles, then no other edges of $G$ are incident to $x$ or $y$. Therefore, $xy$ is contained in at most two triangles.

For each edge $e$ of $G$, let $f(e)$ be the number of triangles containing $e$. So the number of triangles in $G$ is $\frac{1}{3}\sum_{e \in E(G)}f(e) \leq \frac{2}{3}e(G)$. Since $G$ is rainbow-$P_3$-free, we have $e(G) \leq \ex^*(n, P_3) = \frac{3}{2}n$ by Theorem~\ref{path-edge-bounds}. Therefore, $G$ contains at most $\frac{2}{3}\frac{3}{2}n = n$ (rainbow-)$C_3$ copies.
\end{proof}

The proof of $\ex^*(n,P_3)=\frac{3}{2}n$ in \cite{jps} implies that the only rainbow-$P_3$-free graphs attaining $\ex^*(n,P_3)$ are $3$-regular. This can be used to adapt the proof above to also show that the only rainbow-$P_3$-free graphs attaining $\ex^*(n,C_3,P_3)$ are $3$-regular.

\begin{theorem}\label{length-4}
If $n$ is divisible by $8$, then
$\ex^*(n,C_4,P_4) = 3n$. % In general, $\ex^*(n,C_4,P_4) = 3n + O(1)$.
\end{theorem}

\begin{proof}
Theorem~\ref{lower-bound} gives  $3n = \frac{(4-1)!}{2}n \leq  \ex^*(n,C_4,P_4)$. For the upper bound, let $G$ be an $n$-vertex graph with a proper $k$-edge-coloring $c$ with no rainbow-$P_4$. Fix an edge $xy$ of $G$. Without loss of generality, $c(xy) = 1$. We wish to find an upper bound on the number of rainbow-$C_4$ copies containing edge $xy$. We may assume that $xy$ is contained in a rainbow-$C_4$, say $xyzwx$, with edges colored $1,2,3,4$, respectively.

%We claim that every edge $xy$ is contained in at most $3!$ rainbow-$C_4$ copies.
If every rainbow-$C_4$ containing $xy$ has its edges colored from $1,2,3,4$, then it follows from Lemma~\ref{k-edge-colors} that $xy$ is contained in at most $3!$ rainbow-$C_4$ copies. Now suppose that $xyuvx$ is a rainbow-$C_4$ containing $xy$ and exactly one edge is of a color not in  $\{1,2,3,4\}$, say $5$. Associate to this cycle the list $L = (c(xy), c(yu), c(uv), c(vx))$ of its edge colors. We can obtain a different list $L'$ by replacing the entry of color $5$ in $L$ by whichever element of $\{1,2,3,4\}$ is not represented in $L$. 
It is easy to see that $xy$ cannot be in rainbow-$C_4$ copies associated with both lists $L$ and $L'$.
 So if every rainbow-$C_4$ including $xy$ has at most one edge not colored from $\{1,2,3,4\}$, then the list of rainbow-$C_4$ copies containing $xy$ can be put in bijective correspondence with a (possibly proper) subset of the list of all possible rainbow-$C_4$ copies colored from $\{1,2,3,4\}$. Thus, if every rainbow-$C_4$ containing $xy$ has at most one edge not colored from $\{1,2,3,4\}$, then  $xy$ is in at most $3!$ rainbow-$C_4$ copies.

Now suppose (to the contrary) that $xyuvx$ is a rainbow-$C_4$ using two colors not in $\{1,2,3,4\}$, say $5$ and $6$. If an edge of color $5$ or $6$ is incident to exactly one vertex of $xyzw$, then we have a rainbow-$P_4$. Therefore, without loss of generality, we have $u=w$ and $v=z$ and edge colors $c(uw)=5$ and $c(xz)=6$. Any additional edge incident to $xyzw$ forms a rainbow-$P_4$, so there are at most $3!$ rainbow-$C_4$ copies using edge $xy$.

We now count rainbow-$C_4$ copies in $G$ by counting the rainbow-$C_4$ copies on each edge. Let $f(e)$ be the number of rainbow-$C_4$ copies on edge $e$ of $G$. Then, as $G$ is rainbow-$P_4$-free, we have $e(G) \leq \ex^*(n,P_4) = 2n$ by Theorem~\ref{path-edge-bounds}. Therefore, the number of rainbow-$C_4$ copies in $G$ is
\[\frac{1}{4}\sum_{e \in E(G)}f(e) \leq \frac{1}{4}3!e(G) \leq \frac{3!}{4}2n = 3n\]
as desired.
\end{proof}

An unpublished result of Halfpap \cite{Ha} states that the only rainbow-$P_4$-free graphs that attain $\ex^*(n,P_4)$ are $4$-regular. As in the case of $P_3$, this can be used to prove that the only rainbow-$P_4$-free graphs that attain 
 $\ex^*(n,C_4,P_4)$ are also $4$-regular.

Finally, we determine $\ex^*(n,C_5,P_5)$. Note that the proofs of Theorems~\ref{length-3} and \ref{length-4} relied on the bound on $\ex^*(n,P_\ell)$ given in Theorem~\ref{path-edge-bounds}. As $\ex^*(n,P_5)$ is not known exactly, we need a different approach. However, we start in the same way, by bounding the number of rainbow-$C_5$ copies on a fixed edge of a rainbow-$P_5$-free graph.

Throughout the proof of the following lemma and later theorem we will be required to examine many similar cases. Frequently, we will state that it is easy to see that we have a particular edge-coloring. This will involve the inspection of several potential colorings of an individual edge in a given figure and discarding those that lead to either a coloring that is not proper or to a rainbow-$P_5$. It would be excessive to list every possible case, so we leave some of the details to the reader.

\begin{lemma}\label{length-5-lemma}
Let $G$ be an $n$-vertex graph with a proper edge-coloring with no rainbow-$P_5$. Then each edge of $G$ is contained in at most $4!$ rainbow-$C_5$ copies.
\end{lemma}

\begin{proof}
Let us count the number of rainbow-$C_5$ copies in $G$ containing edge  $v_1v_2$. 
We may assume that $v_1v_2$ is in at least one rainbow-$C_5$, say $C=v_1v_2v_3v_4v_5v_1$, whose edges are colored (in order) $1,2,3,4,5$. Note that if every rainbow-$C_5$ containing $v_1v_2$ is colored from $\{1,2,3,4,5\}$, then, by Lemma~\ref{k-edge-colors}, at most $4!$ rainbow-$C_5$ copies in $G$ contain $v_1v_2$. 
An analogous argument to that in the proof of Theorem~\ref{length-4} shows that if every rainbow-$C_5$ containing $v_1v_2$ contains at most one edge not colored from $\{1,2,3,4,5\}$, then at most $4!$ rainbow-$C_5$ copies in $G$ contain $v_1v_2$.

Now suppose that $v_1v_2$ is contained in a rainbow-$C_5$, say $C'$, which contains at least two edges not colored from $\{1,2,3,4,5\}$. We claim that these two edges must be chords of $C$. We write $C' = v_1v_2xyzv_1$, allowing $x,y,$ and $z$ to equal $v_3, v_4,$ or $v_5$. We know that two edges of $C'$ are not colored from $\{1,2,3,4,5\}$; say their colors are $6$ and $7$. We note that to avoid a rainbow-$P_5$, the edges colored $6$ and $7$ must either be chords of $C$ or share no vertices with $C$. So if neither the edge colored $6$ nor the edge colored $7$ is a chord of $C$, then without loss of generality, $c(xy) = 6$, $c(yz) = 7$, and $x,y,$ and $z$ are not equal to $v_3$,$v_4$, or $v_5$. The situation is then as below.

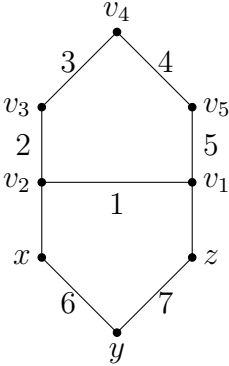
\begin{center}
\begin{tikzpicture}
\filldraw (-1,0) circle(0.05 cm);
\filldraw (1,0) circle(0.05 cm);
\filldraw (-1,1) circle(0.05 cm);
\filldraw (1,1) circle(0.05 cm);
\filldraw (0,2) circle(0.05 cm);
\filldraw (-1,-1) circle(0.05 cm);
\filldraw (1,-1) circle(0.05 cm);
\filldraw (0,-2) circle(0.05 cm);
\draw (-1,0) -- (1,0) node[pos=0.5, below] {1};
\draw (-1,0) -- (-1,1) node[pos=0.5, left] {2};
\draw (1,0) -- (1,1) node[pos=0.5, right] {5};
\draw (-1,1) -- (0,2) node[pos=0.6, left] {3};
\draw (1,1) -- (0,2) node[pos=0.6, right] {4};
\draw (-1,0) -- (-1,-1);
\draw (1,0) -- (1,-1);
\draw (-1,-1) -- (0,-2) node[pos=0.6, left ] {6};
\draw (1,-1) -- (0,-2) node[pos=0.6, right] {7};
\draw (1,0) node[right] {$v_1$};
\draw (-1,0) node[left] {$v_2$};
\draw (-1,1) node[left] {$v_3$};
\draw (0,2) node[above] {$v_4$};
\draw (1,1) node[right] {$v_5$};
\draw (-1,-1) node[left] {$x$};
\draw (0,-2) node[below] {$y$};
\draw (1,-1) node[right] {$z$};
\end{tikzpicture}
\end{center}
It is clear that any choice of $c(v_1z)$ yields a rainbow-$P_5$. So either the edge colored $6$ or the edge colored $7$ is a chord. Without loss of generality, the edge colored $6$ is a chord. Now suppose that the edge colored $7$ is not. It is easy to see that one of the two cases pictured below must occur; dashed edges represent the two possible placements for the chord of color $6$.

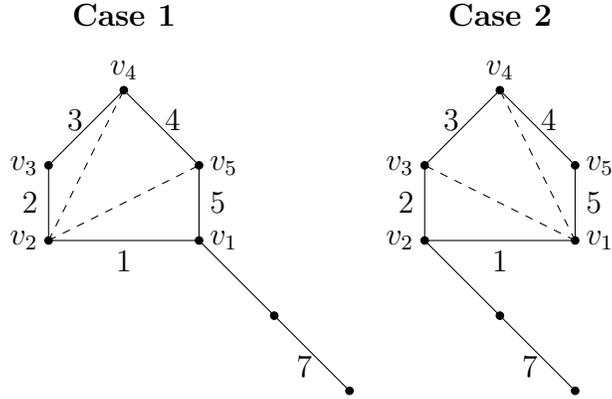
\begin{center}
\begin{tikzpicture}
\filldraw (-1,0) circle(0.05 cm);
\filldraw (1,0) circle(0.05 cm);
\filldraw (-1,1) circle(0.05 cm);
\filldraw (1,1) circle(0.05 cm);
\filldraw (0,2) circle(0.05 cm);
\draw (-1,0) -- (1,0) node[pos=0.5, below] {1};
\draw (-1,0) -- (-1,1) node[pos=0.5, left] {2};
\draw (1,0) -- (1,1) node[pos=0.5, right] {5};
\draw (-1,1) -- (0,2) node[pos=0.6, left] {3};
\draw (1,1) -- (0,2) node[pos=0.6, right] {4};
\draw (1,0) node[right] {$v_1$};
\draw (-1,0) node[left] {$v_2$};
\draw (-1,1) node[left] {$v_3$};
\draw (0,2) node[above] {$v_4$};
\draw (1,1) node[right] {$v_5$};
\filldraw (2,-1) circle (0.05 cm);
\filldraw (3,-2) circle (0.05 cm);
\draw (1,0) -- (2,-1);
%\draw (1,0) -- (2,-1)node[pos=0.4, below] {3};
\draw (2,-1) -- (3,-2)node[pos=0.4, below] {7};
\draw[dashed] (-1,0) -- (1,1);
\draw[dashed] (-1,0) -- (0,2);
\draw (0,3) node{\textbf{Case 1}};

\filldraw (4,0) circle(0.05 cm);
\filldraw (6,0) circle(0.05 cm);
\filldraw (4,1) circle(0.05 cm);
\filldraw (6,1) circle(0.05 cm);
\filldraw (5,2) circle(0.05 cm);
\draw (4,0) -- (6,0) node[pos=0.5, below] {1};
\draw (4,0) -- (4,1) node[pos=0.5, left] {2};
\draw (6,0) -- (6,1) node[pos=0.5, right] {5};
\draw (4,1) -- (5,2) node[pos=0.6, left] {3};
\draw (6,1) -- (5,2) node[pos=0.6, right] {4};
\draw (6,0) node[right] {$v_1$};
\draw (4,0) node[left] {$v_2$};
\draw (4,1) node[left] {$v_3$};
\draw (5,2) node[above] {$v_4$};
\draw (6,1) node[right] {$v_5$};
\filldraw (5, -1) circle (0.05 cm);
\filldraw (6,-2) circle (0.05 cm);
\draw (4,0) -- (5,-1);
%\draw (4,0) -- (5,-1) node[pos=0.4, below]{4};
\draw (5,-1) -- (6,-2) node[pos=0.4, below]{7};
\draw[dashed] (6,0) -- (4,1);
\draw[dashed] (6,0) -- (5,2);
\draw (5,3) node{\textbf{Case 2}};
\end{tikzpicture}
\end{center}

Note that the two cases are analogous by symmetry, so we only need to examine Case 1. Regardless of which choice we make for chord placement, an easy inspection shows that any coloring of the outgoing edge from $v_1$ results in a rainbow-$P_5$. Thus, we conclude that the edges colored $6$ and $7$ are chords of $C$.

We shall now show that $v_1v_2$ is contained in at most $4!$ rainbow-$C_5$ copies.
There are $\binom{5}{2} = 10$ ways in which to place the chords within $C$; up to symmetry, six are distinct. They are pictured below.

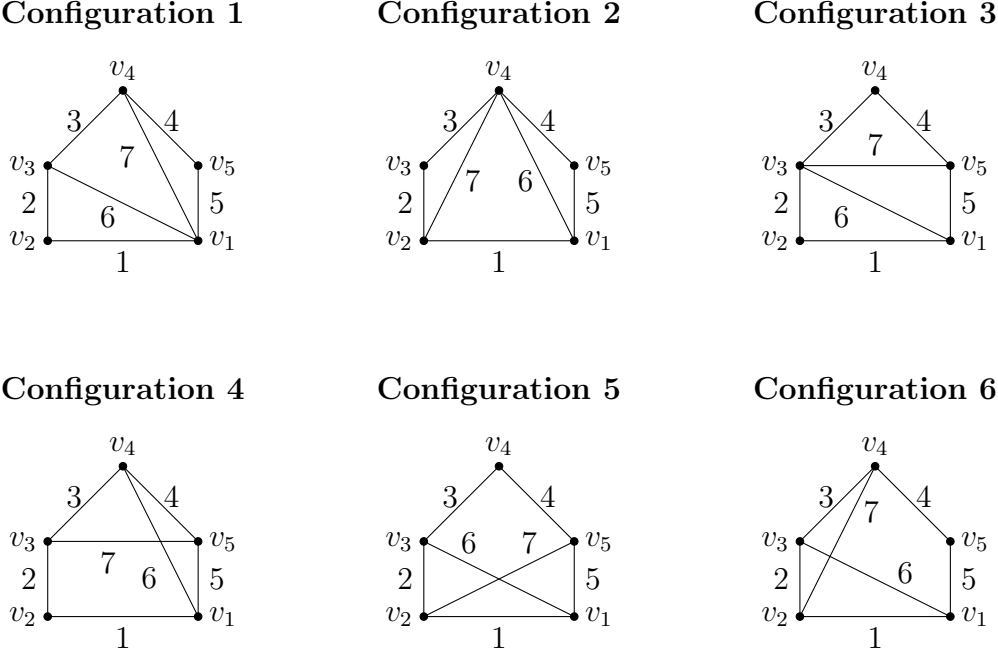
\begin{center}
\begin{tikzpicture}
\filldraw (-1,0) circle(0.05 cm);
\filldraw (1,0) circle(0.05 cm);
\filldraw (-1,1) circle(0.05 cm);
\filldraw (1,1) circle(0.05 cm);
\filldraw (0,2) circle(0.05 cm);
\draw (-1,0) -- (1,0) node[pos=0.5, below] {1};
\draw (-1,0) -- (-1,1) node[pos=0.5, left] {2};
\draw (1,0) -- (1,1) node[pos=0.5, right] {5};
\draw (-1,1) -- (0,2) node[pos=0.6, left] {3};
\draw (1,1) -- (0,2) node[pos=0.6, right] {4};
\draw (1,0) node[right] {$v_1$};
\draw (-1,0) node[left] {$v_2$};
\draw (-1,1) node[left] {$v_3$};
\draw (0,2) node[above] {$v_4$};
\draw (1,1) node[right] {$v_5$};
\draw (1,0) -- (-1,1) node[pos=0.6, below]{6};
\draw (1,0) -- (0,2) node[pos=0.7, below left]{7};
\draw (0,3) node{\textbf{Configuration 1}};

\filldraw (4,0) circle(0.05 cm);
\filldraw (6,0) circle(0.05 cm);
\filldraw (4,1) circle(0.05 cm);
\filldraw (6,1) circle(0.05 cm);
\filldraw (5,2) circle(0.05 cm);
\draw (4,0) -- (6,0) node[pos=0.5, below] {1};
\draw (4,0) -- (4,1) node[pos=0.5, left] {2};
\draw (6,0) -- (6,1) node[pos=0.5, right] {5};
\draw (4,1) -- (5,2) node[pos=0.6, left] {3};
\draw (6,1) -- (5,2) node[pos=0.6, right] {4};
\draw (6,0) node[right] {$v_1$};
\draw (4,0) node[left] {$v_2$};
\draw (4,1) node[left] {$v_3$};
\draw (5,2) node[above] {$v_4$};
\draw (6,1) node[right] {$v_5$};
\draw (4,0) -- (5,2) node[pos=0.4, right]{7};
\draw (6,0) -- (5,2) node[pos=0.4, left]{6};
\draw (5,3) node{\textbf{Configuration 2}};

\filldraw (9,0) circle(0.05 cm);
\filldraw (11,0) circle(0.05 cm);
\filldraw (9,1) circle(0.05 cm);
\filldraw (11,1) circle(0.05 cm);
\filldraw (10,2) circle(0.05 cm);
\draw (9,0) -- (11,0) node[pos=0.5, below] {1};
\draw (9,0) -- (9,1) node[pos=0.5, left] {2};
\draw (11,0) -- (11,1) node[pos=0.5, right] {5};
\draw (9,1) -- (10,2) node[pos=0.6, left] {3};
\draw (11,1) -- (10,2) node[pos=0.6, right] {4};
\draw (11,0) node[right] {$v_1$};
\draw (9,0) node[left] {$v_2$};
\draw (9,1) node[left] {$v_3$};
\draw (10,2) node[above] {$v_4$};
\draw (11,1) node[right] {$v_5$};
\draw (11,0) -- (9,1) node[pos=0.6, below left]{6};
\draw (9,1) -- (11,1) node[pos=0.5, above]{7};
\draw (10,3) node{\textbf{Configuration 3}};

\filldraw (-1,-5) circle(0.05 cm);
\filldraw (1,-5) circle(0.05 cm);
\filldraw (-1,-4) circle(0.05 cm);
\filldraw (1,-4) circle(0.05 cm);
\filldraw (0,-3) circle(0.05 cm);
\draw (-1,-5) -- (1,-5) node[pos=0.5, below] {1};
\draw (-1,-5) -- (-1,-4) node[pos=0.5, left] {2};
\draw (1,-5) -- (1,-4) node[pos=0.5, right] {5};
\draw (-1,-4) -- (0,-3) node[pos=0.6, left] {3};
\draw (1,-4) -- (0,-3) node[pos=0.6, right] {4};
\draw (1,-5) node[right] {$v_1$};
\draw (-1,-5) node[left] {$v_2$};
\draw (-1,-4) node[left] {$v_3$};
\draw (0,-3) node[above] {$v_4$};
\draw (1,-4) node[right] {$v_5$};
\draw (1,-5) -- (0,-3) node[pos=0.4, below left]{6};
\draw (-1,-4) -- (1,-4) node[pos=0.4, below]{7};
\draw (0,-2) node{\textbf{Configuration 4}};

\filldraw (4,-5) circle(0.05 cm);
\filldraw (6,-5) circle(0.05 cm);
\filldraw (4,-4) circle(0.05 cm);
\filldraw (6,-4) circle(0.05 cm);
\filldraw (5,-3) circle(0.05 cm);
\draw (4,-5) -- (6,-5) node[pos=0.5, below] {1};
\draw (4,-5) -- (4,-4) node[pos=0.5, left] {2};
\draw (6,-5) -- (6,-4) node[pos=0.5, right] {5};
\draw (4,-4) -- (5,-3) node[pos=0.6, left] {3};
\draw (6,-4) -- (5,-3) node[pos=0.6, right] {4};
\draw (6,-5) node[right] {$v_1$};
\draw (4,-5) node[left] {$v_2$};
\draw (4,-4) node[left] {$v_3$};
\draw (5,-3) node[above] {$v_4$};
\draw (6,-4) node[right] {$v_5$};
\draw (4,-5) -- (6,-4) node[pos=0.7, above]{7};
\draw (6,-5) -- (4,-4) node[pos=0.7, above]{6};
\draw (5,-2) node{\textbf{Configuration 5}};

\filldraw (9,-5) circle(0.05 cm);
\filldraw (11,-5) circle(0.05 cm);
\filldraw (9,-4) circle(0.05 cm);
\filldraw (11,-4) circle(0.05 cm);
\filldraw (10,-3) circle(0.05 cm);
\draw (9,-5) -- (11,-5) node[pos=0.5, below] {1};
\draw (9,-5) -- (9,-4) node[pos=0.5, left] {2};
\draw (11,-5) -- (11,-4) node[pos=0.5, right] {5};
\draw (9,-4) -- (10,-3) node[pos=0.6, left] {3};
\draw (11,-4) -- (10,-3) node[pos=0.6, right] {4};
\draw (11,-5) node[right] {$v_1$};
\draw (9,-5) node[left] {$v_2$};
\draw (9,-4) node[left] {$v_3$};
\draw (10,-3) node[above] {$v_4$};
\draw (11,-4) node[right] {$v_5$};
\draw (11,-5) -- (9,-4) node[pos=0.3, above]{6};
\draw (9,-5) -- (10,-3) node[pos=0.7, right]{7};
\draw (10,-2) node{\textbf{Configuration 6}};
\end{tikzpicture}
\end{center}

We need not consider Configuration $1$ or $2$, since it is clear that chords placed in these configurations cannot form a $C_5$ containing $v_1v_2$. In the remaining four configurations, we will show that $v_1v_2$ is contained in at most $4!$ rainbow-$C_5$ copies. In these arguments, we will use the fact that $v_1v_2$ is contained in at most $3!$ rainbow-$C_5$ copies which contain only vertices from $C$ (since there are $3!$ ways to permute $v_3,v_4,$ and $v_5$). We will also require the observation that given five vertices and three fixed edges among them, there are (at most) two ways in which to add another two edges to create a $C_5$. We also repeatedly use the fact that there is no edge with one vertex incident to $C$ that is colored with a color not in $\{1,2,3,4,5\}$ as this results in a rainbow-$P_5$.

We first consider Configuration $3$. If $v_1v_2$ is on a rainbow-$C_5$ containing both of the pictured chords, then this $C_5$ is of the form $v_2v_1v_3v_5uv_2$, where $u$ is either equal to $v_4$ or to some vertex not on $C$. If $u \neq v_4$, then the situation is as pictured below

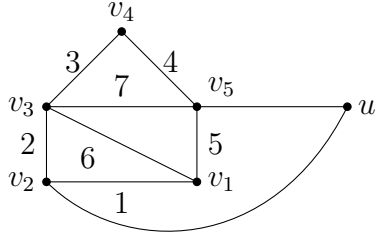
\begin{center}
\begin{tikzpicture}
\filldraw (9,0) circle(0.05 cm);
\filldraw (11,0) circle(0.05 cm);
\filldraw (9,1) circle(0.05 cm);
\filldraw (11,1) circle(0.05 cm);
\filldraw (10,2) circle(0.05 cm);
\draw (9,0) -- (11,0) node[pos=0.5, below] {1};
\draw (9,0) -- (9,1) node[pos=0.5, left] {2};
\draw (11,0) -- (11,1) node[pos=0.5, right] {5};
\draw (9,1) -- (10,2) node[pos=0.6, left] {3};
\draw (11,1) -- (10,2) node[pos=0.6, right] {4};
\draw (11,0) node[right] {$v_1$};
\draw (9,0) node[left] {$v_2$};
\draw (9,1) node[left] {$v_3$};
\draw (10,2) node[above] {$v_4$};
\draw (11,1) node[above right] {$v_5$};
\draw (11,0) -- (9,1) node[pos=0.6, below left]{6};
\draw (9,1) -- (11,1) node[pos=0.5, above]{7};
\filldraw (13,1) circle(0.05 cm);
\draw (13,1) node[right]{$u$};
\draw (11,1) -- (13,1);
\draw (13,1) .. controls (12,-1) and (10,-1) ..(9,0);
\end{tikzpicture}
\end{center}
Because our coloring must be proper, $c(v_2u)$ is not $1$ or $2$. In order to avoid a rainbow-$P_5$, it is clear that $c(v_2u)$ must be in $\{1,2,3,4,5\}$. However, if $c(v_2u) \in \{3,4,5\}$, then either $uv_2v_1v_3v_5v_4$ or $uv_2v_1v_5v_3v_4$ is a rainbow-$P_5$. So we must have $u = v_4$. Thus, if the chords placed in Configuration $3$ yield a rainbow-$C_5$ containing $v_1v_2$, then that rainbow-$C_5$ is $v_2v_1v_3v_5v_4v_2$, as drawn below.

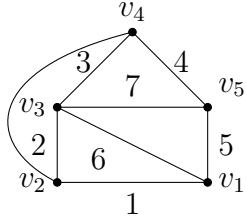
\begin{center}
\begin{tikzpicture}
\filldraw (9,0) circle(0.05 cm);
\filldraw (11,0) circle(0.05 cm);
\filldraw (9,1) circle(0.05 cm);
\filldraw (11,1) circle(0.05 cm);
\filldraw (10,2) circle(0.05 cm);
\draw (9,0) -- (11,0) node[pos=0.5, below] {1};
\draw (9,0) -- (9,1) node[pos=0.5, left] {2};
\draw (11,0) -- (11,1) node[pos=0.5, right] {5};
\draw (9,1) -- (10,2) node[pos=0.6, left] {3};
\draw (11,1) -- (10,2) node[pos=0.6, right] {4};
\draw (11,0) node[right] {$v_1$};
\draw (9,0) node[left] {$v_2$};
\draw (9,1) node[left] {$v_3$};
\draw (10,2) node[above] {$v_4$};
\draw (11,1) node[above right] {$v_5$};
\draw (11,0) -- (9,1) node[pos=0.6, below left]{6};
\draw (9,1) -- (11,1) node[pos=0.5, above]{7};
\draw (10,2) .. controls (8,1.5) and (8,.5) ..(9,0);
\end{tikzpicture}
\end{center}

Note that to ensure that the coloring is proper and that $v_2v_1v_3v_5v_4v_2$ is a rainbow-$C_5$, we must have $c(v_2v_4) = 5$ or $c(v_2v_4)$ is a color not yet used, say $8$. Now, inspect $v_1, v_2, v_4,$ and $v_5$. It is easy (but somewhat tedious) to check that none of these vertices may be adjacent to any vertex $u$ which is not on $C$; any color choice for such an edge will result in a rainbow-$P_5$ given the above configuration and regardless of whether $c(v_2v_4)$ is chosen to equal $5$ or $8$. Thus, $v_1v_2$ lies in no cycle except those using only vertices from $C$. Hence, $v_1v_2$ is contained in at most $3! < 4!$ rainbow-$C_5$ copies.

In Configuration $4$, we observe that $v_2$ and $v_5$ can only be adjacent to vertices on $C$, 
and that if $v_4$ is incident to an edge whose other endpoint is not on $C$, then that edge must be colored $1$. So if a rainbow-$C_5$ contains $v_1v_2$ and uses vertices not on $C$, then it must include edges of the form $v_1u$ and $v_3w$ where $u$ and $w$ are not on $C$ (although we allow $u = w$). We observe that $c(v_3w)$ cannot equal $5$, so we must have $c(v_3w) = 4$ if the cycle is to be rainbow. This forces $c(v_1u) = 2$, since $c(v_1u)$ cannot equal $3$. The number of rainbow-$C_5$ copies containing $v_1u$ and $v_3w$ is at most $2$, since if $u \neq w$, then we have specified all five vertices of the cycle and three of its edges, so there are only two ways to add the remaining two edges. If $u = w$, then the fifth vertex of the cycle must be on $C$, and there are two choices for this vertex. So in configuration $4$, $v_1v_2$ is on at most $3! + 2 < 4!$ rainbow-$C_5$ copies.

In Configuration $5$, we observe that $v_1$, $v_2$, and $v_4$ are adjacent only to vertices on $C$. Also, if $v_3u$ is an edge with $u$ not on $C$, then $c(v_3u) \in \{1,4\}$, and if $v_5w$ is an edge with $w$ not on $C$, then $c(v_5w) \in \{1,3\}$. So the only possible rainbow-$C_5$ copies containing $v_1v_2$ and vertices not on $C$ contain edges $v_3u$ and $v_5w$ with $c(v_3u) = 4$ and $c(v_5w) = 3$. Note that, in order to have exactly five vertices, we must have $u = w$. We have now specified all five vertices and three edges of a cycle, so there are at most two ways to add edges to create a $C_5$. Hence, there are at most $3! + 2 < 4!$ containing $v_1v_2$ in Configuration $5$.

In Configuration $6$, we note that $v_2$, $v_3$, and $v_5$ are only adjacent to vertices on $C$. If $v_1u$ is an edge with $u$ not on $C$, then $c(v_1u) \in \{2,4\}$, and if $v_4w$ is an edge with $w$ not on $C$, then $c(v_4w) \in \{2,5\}$. Thus, the only possible rainbow-$C_5$ copies containing $v_1v_2$ and some vertex not on $C$ use a pair of edges $v_1u$ and $v_4w$. Since each of $v_1,v_4$ can have at most two neighbors not on $C$, and at most two cycles can be formed which include $v_1v_2$ and a fixed pair of edges $v_1u$ and $v_4w$, the edge $v_1v_2$ is contained in at most $3! + 8 < 4!$ rainbow-$C_5$ copies.

Thus, if $v_1v_2$ is contained in a rainbow-$C_5$ which uses two colors not in $\{1,2,3,4,5\}$, then $v_1v_2$ is contained in strictly fewer than $4!$ rainbow-$C_5$ copies.
\end{proof}

We may immediately apply Lemma~\ref{length-5-lemma} to get 
$\ex^*(n, C_5, P_5) \leq \frac{4!}{5} \ex^*(n,P_5)$. If we could show that $\ex^*(n,P_5)$ was $\frac{5}{2}n$, then Lemma~\ref{length-5-lemma} would give the desired bound on $\ex^*(n,C_5,P_5)$. Unfortunately, this is not known. However, we can give a new upper bound on $\ex^*(n,P_5)$ which combined with Lemma~\ref{length-5-lemma} gives $\ex^*(n,C_5,P_5) \leq \frac{4!}{5} \cdot 4n = 19.2n$.

\begin{theorem}
$\ex^*(n,P_5) \leq 4n$.
\end{theorem}
\begin{proof}
Let $G$ be an $n$-vertex graph with a proper edge-coloring and more than $4n$ edges. We will show that $G$ contains a rainbow-$P_5$.
The average degree of $G$ is greater than $8$. By removing low degree vertices, we can obtain a subgraph $G'$ of $G$ with minimum degree at least $5$ and average degree greater than $8$. In particular, $G'$ has a vertex, say $v$, of degree at least $9$.

\smallskip

\textbf{Case 1: }  $G'$ contains a rainbow-$P_4$ ending at $v$. 

\smallskip

Let $P=vxyzw$ be a rainbow-$P_4$ ending at $v$. Since $d(v) \geq 9$, $v$ must be adjacent to at least $5$ vertices not on $P$. Since the coloring of $G'$ is proper, none of these five edges receives the same color as $vx$. Three may receive the colors used for $xy$, $yz$, and $zw$, but two must receive colors not used in $P$. Either of these two edges will extend $P$ to a rainbow-$P_5$.

\smallskip

\textbf{Case 2:}  $G'$ does not contain a rainbow-$P_4$ ending at $v$.

\smallskip

Using the fact that the minimum degree in $G'$ is at least $5$, we can greedily build a rainbow path of length $3$ ending at $v$; moreover, since this path does not extend to a rainbow-$P_4$, then the situation must be as pictured below:

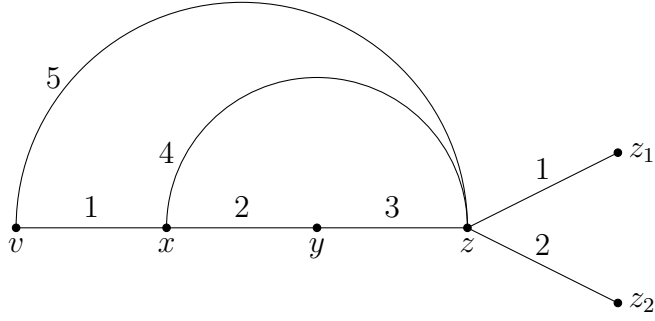
\begin{center}
\begin{tikzpicture}
\filldraw (0,0) circle(0.05 cm) node[below]{$v$};
\filldraw (2,0) circle(0.05 cm) node[below]{$x$};
\filldraw (4,0) circle(0.05 cm) node[below]{$y$};
\filldraw (6,0) circle(0.05 cm) node[below]{$z$};
\draw (0,0) -- (6,0);
\draw (1, 0) node[above]{$1$};
\draw (3, 0) node[above]{$2$};
\draw (5, 0) node[above]{$3$};
\filldraw (8,1) circle(0.05 cm) node[right]{$z_1$};
\filldraw (8,-1) circle(0.05 cm) node[right]{$z_2$};
\draw (6,0) -- (8,1) node[pos=0.5, above]{1};
\draw (6,0) -- (8,-1)node[pos=0.5, above]{2};
\draw (6,0) arc (0:180:2);
\draw (6,0) arc (0:180:3);
\draw (2,1) node{$4$};
\draw (0.5,2) node{$5$};
\end{tikzpicture}
\end{center}

Consider the vertex $y$. Since $d(y) \geq 5$, $y$ must be adjacent to at least two vertices not on $vxyz$. Call these $y_1$ and $y_2$ (we allow that $y_1$ and $y_2$ may not be distinct from $z_1$ and $z_2$). It is easy to see that if $c(yy_i)$ is not $2$ or $4$, then either $vzxyy_i$ or $vxzyy_i$ is a rainbow-$P_4$ ending in $v$, a contradiction. Moreover, $c(yy_i) \neq 2$ or the coloring is not proper. So both $c(yy_1)$ and $c(yy_2)$ must be $4$, a contradiction.
\end{proof}

%By definition, $c(yy_1) \neq 2$ (since $c(xy) = 2$), so $c(yy_1)$ must equal $4$ or $5$, and analogously for $c(yy_2)$. So without loss of generality, $c(yy_1) = 4$ and $c(yy_2) = 5$. Thus, in particular, there is not a third edge from $y$ to a vertex not on $vxyz$, since this edge then cannot be colored $4$ or $5$, and by the above argument will create a rainbow $P_4$ ending in $v$. So, $y$ must be adjacent to exactly two vertices not on $vxyz$; combined with the degree condition on $y$, this implies that $vy$ is an edge. We draw $vy$ below, and note that since $yy_1$ (not pictured) receives color $4$, $c(vy)$ must equal a color not yet used, say $6$.

%\begin{center}
%\begin{tikzpicture}
%\filldraw (0,0) circle(0.05 cm) node[below left]{$v$};
%\filldraw (2,0) circle(0.05 cm) node[below]{$x$};
%\filldraw (4,0) circle(0.05 cm) node[below right]{$y$};
%\filldraw (6,0) circle(0.05 cm) node[below]{$z$};
%\draw (0,0) -- (6,0);
%\draw (1, 0) node[above]{$1$};
%\draw (3, 0) node[above]{$2$};
%\draw (5, 0) node[above]{$3$};
%\filldraw (8,1) circle(0.05 cm) node[right]{$z_1$};
%\filldraw (8,-1) circle(0.05 cm) node[right]{$z_2$};
%\draw (6,0) -- (8,1) node[pos=0.5, above]{1};
%\draw (6,0) -- (8,-1)node[pos=0.5, above]{2};
%\draw (6,0) arc (0:180:2);
%\draw (6,0) arc (0:180:3);
%\draw (2,1) node{$4$};
%\draw (0.5,2) node{$5$};
%\draw (4,0) arc (0:-180:2);
%\draw (4,-1) node{$6$};
%\end{tikzpicture}
%\end{center}

%But now $vyxzz_1$ is a rainbow $P_4$ ending at $v$. 

%\end{proof}

On the other hand, by Lemma~\ref{degree-lemma}, any vertex in a rainbow-$C_5$ in a rainbow-$P_5$-free graph has degree at most $7$. We can count rainbow-$C_5$ copies on a fixed vertex as follows. For every vertex $v$ which is on a rainbow-$C_5$, each rainbow-$C_5$ containing $v$ begins with an edge incident to $v$. There are at most $7$ choices of edge, and each is contained in at most $4!$ rainbow-$C_5$ copies by Lemma~\ref{length-5-lemma}. Each rainbow-$C_5$ is counted ten times this way as each $C_5$ contains five vertices and each rainbow-$C_5$ is counted twice per vertex (because every rainbow-$C_5$ containing $v$ in fact uses two edges incident to $v$, and is counted once by each). In this way we obtain the following slight improvement $\ex^*(n,C_5,P_5) \leq \frac{4!}{5}\cdot\frac{7}{2}n = 16.8n$.

The bounds given above are clearly not the best possible; it is easy to show that if $G$ is a rainbow-$P_5$-free graph containing a rainbow-$C_5$, say $C$, then not every vertex on $C$ can have degree $7$. Therefore, a more careful analysis of degree constraints for vertices on a rainbow-$C_5$ is needed.

The proof of our upper bound relies on Lemma~\ref{length-5-lemma} and another key step. We show that a rainbow-$C_5$ containing high-degree vertices must contain vertices of low degree.
By appropriately pairing vertices of high degree and low degree we can show
that the average degree over all vertices contained in a rainbow-$C_5$ is at most $5$. Combining this observation with Lemma~\ref{length-5-lemma} will give the desired bound on $\ex^*(n,C_5,P_5)$.

\begin{theorem}\label{length-5}
If $n$ is divisible by $16$, then
$\ex^*(n, C_5, P_5) = 12 n$.% when $16$ divides $n$. In general, $\ex^*(n, C_5, P_5) = 12 n + O(1)$.
\end{theorem}

\begin{proof}
Theorem~\ref{lower-bound} gives $12 n = \frac{(5-1)!}{2}n \leq \ex^*(n, C_5, P_5)$. To prove the upper bound, consider an $n$-vertex graph $G$ with a proper edge-coloring with no rainbow-$P_5$. Let $V'$ be the set of vertices in $G$ which are contained in at least one rainbow-$C_5$. By Lemma \ref{length-5-lemma}, any vertex $v$ in $V'$ is contained in at most $\frac{4!d(v)}{2}$ rainbow-$C_5$ copies, since each edge incident to $v$ is in at most $4!$ rainbow-$C_5$ copies and each rainbow-$C_5$ containing $v$ uses two edges incident to $v$. Thus, the total number of rainbow-$C_5$ copies in $G$ is at most 
$$\underset{v \in V'}{\sum} \frac{4!d(v)}{2 \cdot 5} = \frac{4!}{2 \cdot 5} \underset{v \in V'}{\sum} d(v).$$
If the average degree of vertices in $V'$ is at most $5$, then we immediately have 
$$\frac{4!}{2 \cdot 5} \underset{v \in V'}{\sum} d(v) \leq \frac{4!}{2 \cdot 5} \cdot 5 |V'| \leq \frac{4!}{2} n = 12 n,$$
and we are done. In order to establish that the average degree in $V'$ is at most $5$, we will need the following technical claim.

\begin{claim}\label{degree-claim}
Let $C$ be a rainbow-$C_5$ in $G$ containing a vertex of degree at least $6$ and let $S$ be the set of vertices on $C$ with degree at least $6$. Then there is a set $T$ of vertices on $C$ such that:

\begin{enumerate}

\item[(1)]  each vertex of $S$ is adjacent to at least one vertex of $T$, and each vertex of $S$ is adjacent to at least one vertex of $T$;

\item[(2)]  if $v \in T$ is adjacent to $u$ and $d(u) \geq 6$, then $u \in S$;

\item[(3)] $\displaystyle \frac{1}{|T|+ |S|} \sum_{v \in T \cup S}d(v) \leq 5$.

\end{enumerate}

\end{claim}

\begin{proof}

We call a pair of sets $S, T$ satisfying all of the above conditions an \textit{S,T pair}.

Without loss of generality $C = v_1v_2v_3v_4v_5v_1$ has edges colored (in order) $1,2,3,4,5$, and $d(v_1) > 5$. By Lemma \ref{degree-lemma}, $d(v_1) \leq 7$, so we must either have $d(v_1) = 6$ or $d(v_1) = 7$. We shall consider both cases. Frequently in the cases below, we shall recall the following simple observation: If any vertex $v_i$ of $C$ is incident to an edge which is not colored from $\{ 1, 2, 3, 4, 5\}$, then this edge is of the form $v_iv_j$ for some vertex $v_j$ of $C$. In particular, if $d(v_i) = 5+ k$, then $v_i$ is incident to at least $k$ chords of $C$ whose colors are not in $\{1,2,3,4,5\}$. Recall that there is no edge of color not in $\{1,2,3,4,5\}$ with exactly one endpoint in $C$ as otherwise we get a rainbow-$P_5$.

\smallskip

\textbf{Case 1:} $d(v_1) = 7$. 

\smallskip

Since $d(v_1) = 7$, $v_1$ has three neighbors not on $C$, say $u_1,u_2,$ and $u_3$, and both $v_1v_3$ and $v_1v_4$ are edges. Without loss of generality, we have $c(v_1u_1) = 2$, $c(v_1u_2) = 3$, $c(v_1u_3) = 4$, $c(v_1v_3) = 6$, and $c(v_1v_4) = 7$.

We first bound $d(v_2)$. It is easy to check that any edge $v_2 w$ with $w$ not on $C$ creates a rainbow-$P_5$. Also, if $v_2v_4$ is an edge, then (noting that $c(v_2v_4)$ is not equal to $1,2,3,4$, or $7$) either $u_2v_1v_3v_2v_4v_5$ or $u_2v_1v_5v_4v_2v_3$ is a rainbow-$P_5$. Hence, $d(v_2) \leq 3$. By symmetry, $d(v_5) \leq 3$.

We next bound $d(v_4)$. As noted above, $v_4v_2$ is not an edge. Also, if $v_4w$ is an edge with $w$ not on $C$, then $c(v_4w) \neq 1$, since otherwise $wv_4v_5v_1v_3v_2$ is rainbow. So $v_4$ has at most two neighbors not on $C$ (since the edge incident to any such neighbor must be colored either $2$ or $5$), and at most three neighbors on $C$. Hence, $d(v_4) \leq 5$. By symmetry, $d(v_3) \leq 5$.
Thus, in this case, we can choose $S= \{v_1\}$ and $T = \{v_2\}$ to form an $S,T$ pair.

\smallskip

 \textbf{Case 2:} $d(v_1) = 6$.

\smallskip

 We distinguish three subcases.
 
 \smallskip

 \textbf{Case 2.1:} $v_1$ is adjacent to both $v_3$ and $v_4,$ and both $c(v_1v_4)$ and $c(v_1v_4)$ are not in $\{1,2,3,4,5\}$.
 
 \smallskip

Without loss of generality, $c(v_1v_3) = 6$ and $c(v_1v_4) = 7$. We observe that $v_2$ and $v_5$ are only adjacent to vertices on $C$, so have degrees at most $4$. Moreover, suppose that one of $v_3,v_4$ has degree greater than $5$. The two vertices are symmetric thus far, so we may assume that $d(v_3) > 5$. We established in Case 1 that a vertex of degree $7$ is never on a rainbow-$C_5$ containing any other vertex of degree greater than $5$, so $d(v_3) = 6$. We note that if $v_3u$ is an edge with $u$ not on $C$, then $c(v_3u) \neq 5$ (as otherwise we get a rainbow-$P_5$). The picture then must be as below.

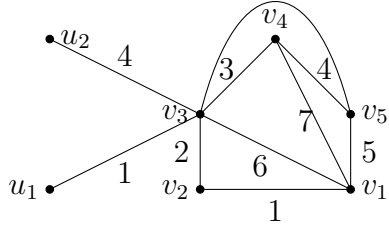
\begin{center}
\begin{tikzpicture}
\filldraw (-1,0) circle(0.05 cm);
\filldraw (1,0) circle(0.05 cm);
\filldraw (-1,1) circle(0.05 cm);
\filldraw (1,1) circle(0.05 cm);
\filldraw (0,2) circle(0.05 cm);
\draw (-1,0) -- (1,0) node[pos=0.5, below] {1};
\draw (-1,0) -- (-1,1) node[pos=0.5, left] {2};
\draw (1,0) -- (1,1) node[pos=0.5, right] {5};
\draw (-1,1) -- (0,2) node[pos=0.6, left] {3};
\draw (1,1) -- (0,2) node[pos=0.6, right] {4};
\draw (1,0) node[right] {$v_1$};
\draw (-1,0) node[left] {$v_2$};
\draw (-1,1) node[left] {$v_3$};
\draw (0,2) node[above] {$v_4$};
\draw (1,1) node[right] {$v_5$};
\draw (1,0) -- (-1,1) node[pos=0.6, below]{6};
\draw (1,0) -- (0,2) node[pos=0.6, below]{7};
\filldraw (-3, 0) circle (0.05 cm);
\filldraw (-3, 2) circle (0.05 cm);
\draw (-3,0) node[left] {$u_1$};
\draw (-3,2) node[right] {$u_2$};
\draw (-1,1) -- (-3, 0) node[pos=0.5, below] {1};
\draw (-1,1) -- (-3, 2) node[pos=0.5, above] {4};
\draw (-1,1) .. controls (-.5,3) and (.5,3) .. (1,1);
\end{tikzpicture}
\end{center}

In order for the coloring to be proper, $c(v_3v_5)$ is either $7$ or a color not yet used, say $8$. With this fact, we may observe that $v_4$ is also adjacent only to vertices on $C$. 

Thus, we may choose either $S = \{v_1\}$ and $T = \{v_2\}$ or $S = \{v_1, v_3\}$ and $T = \{v_2,v_4\}$ to obtain an $S,T$ pair.

\smallskip

\textbf{Case 2.2:} $v_1$ is adjacent to both $v_3$ and $v_4,$ and one of $c(v_1v_3), c(v_1v_4)$ is in $\{1,2,3,4,5\}$.

\smallskip

Without loss of generality, $c(v_1v_3)$ is not in $\{1,2,3,4,5\}$, say $c(v_1v_3) = 6$. 
So $c(v_1v_4) \in \{1,2,3,4,5\}$, which forces $c(v_1v_4) = 2$. The vertex $v_1$ is adjacent to two vertices not on $C$, say $w_1$ and $w_2$. Without loss of generality, $c(v_1w_1) = 3$ and $c(v_1w_2) = 4$. We draw this below.

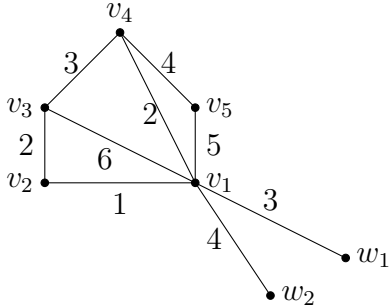
\begin{center}
\begin{tikzpicture}
\filldraw (-1,0) circle(0.05 cm);
\filldraw (1,0) circle(0.05 cm);
\filldraw (-1,1) circle(0.05 cm);
\filldraw (1,1) circle(0.05 cm);
\filldraw (0,2) circle(0.05 cm);
\draw (-1,0) -- (1,0) node[pos=0.5, below] {1};
\draw (-1,0) -- (-1,1) node[pos=0.5, left] {2};
\draw (1,0) -- (1,1) node[pos=0.5, right] {5};
\draw (-1,1) -- (0,2) node[pos=0.6, left] {3};
\draw (1,1) -- (0,2) node[pos=0.6, right] {4};
\draw (1,0) node[right] {$v_1$};
\draw (-1,0) node[left] {$v_2$};
\draw (-1,1) node[left] {$v_3$};
\draw (0,2) node[above] {$v_4$};
\draw (1,1) node[right] {$v_5$};
\draw (1,0) -- (-1,1) node[pos=0.6, below]{6};
\draw (1,0) -- (0,2) node[pos=0.6, below]{2};
\filldraw (3,-1) circle (0.05 cm);
\filldraw (2,-1.5) circle (0.05 cm);
\draw (3,-1) node[right] {$w_1$};
\draw (2,-1.5) node[right] {$w_2$};
\draw (1,0) -- (3,-1) node[pos = 0.5, above] {3};
\draw (1,0) -- (2,-1.5) node[pos = 0.5, left] {4};
\end{tikzpicture}
\end{center}

Observe that if the edge $v_2v_4$ is present, then $c(v_2v_4)$ is not in $\{1,2,3,4\}$, and so either $w_1v_1v_5v_4v_2v_3$ or $w_1v_1v_3v_2v_4v_5$ is a rainbow-$P_5$. Thus, $v_2v_4$ is not an edge. Furthermore, if the edge $v_2v_5$ is present, then $c(v_2v_5)$ is not in $\{1,2,4,5\}$, and cannot be $6$ or $7$, since then $v_4v_3v_2v_5v_1w_2$ is rainbow. So if $v_2v_5$ is an edge, then $c(v_2v_5) = 3$. Finally, if $v_2u$ is an edge with $u$ not on $C$, then $c(v_2u)$ must be in $4$, else one of $uv_2v_3v_1v_5v_4$ or $uv_2v_1v_3v_4v_5$ is rainbow. Thus, $d(v_2) \leq 4$.

We have seen that $v_2v_4$ is not an edge. Observe that if $v_4u$ is an edge with $u$ not on $C$, then $c(v_4,u)$ cannot be in $\{2,3,4\}$, which forces $c(v_4u) = 5$, else either $uv_4v_5v_1v_2v_3$ or $uv_4v_5v_1v_3v_2$ is rainbow. So $d(v_4) \leq 4$.

Finally, suppose that $d(v_5) \geq 6$. So $v_5$ must have an incident edge of color not in $\{1,2,3,4,5\}$. This edge must have both endpoints in $C$. We have seen that if $v_2v_5$ is an edge it is color $3$, so $v_3v_5$ must be this edge. The vertex $v_3$ is incident to an edge of color $6$, so $v_5v_3$ must be a color not yet used, say $7$.
Now observe that $w_1v_1v_2v_3v_5v_4$ is rainbow, a contradiction.
We illustrate this below.

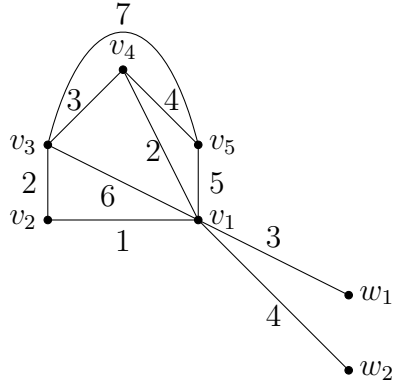
\begin{center}
\begin{tikzpicture}
\filldraw (-1,0) circle(0.05 cm);
\filldraw (1,0) circle(0.05 cm);
\filldraw (-1,1) circle(0.05 cm);
\filldraw (1,1) circle(0.05 cm);
\filldraw (0,2) circle(0.05 cm);
\draw (-1,0) -- (1,0) node[pos=0.5, below] {1};
\draw (-1,0) -- (-1,1) node[pos=0.5, left] {2};
\draw (1,0) -- (1,1) node[pos=0.5, right] {5};
\draw (-1,1) -- (0,2) node[pos=0.6, left] {3};
\draw (1,1) -- (0,2) node[pos=0.6, right] {4};
\draw (1,0) node[right] {$v_1$};
\draw (-1,0) node[left] {$v_2$};
\draw (-1,1) node[left] {$v_3$};
\draw (0,2) node[above] {$v_4$};
\draw (1,1) node[right] {$v_5$};
\draw (1,0) -- (-1,1) node[pos=0.6, below]{6};
\draw (1,0) -- (0,2) node[pos=0.6, below]{2};
\filldraw (3,-1) circle (0.05 cm);
\filldraw (3,-2) circle (0.05 cm);
\draw (3,-1) node[right] {$w_1$};
\draw (3,-2) node[right] {$w_2$};
\draw (1,0) -- (3,-1) node[pos = 0.5, above] {3};
\draw (1,0) -- (3,-2) node[pos = 0.5, below] {4};
\draw (1,1) .. controls (.5,3) and (-.5,3) .. (-1,1);
%\draw (1,1) .. controls (4,-.5) and (4,-1.5) .. (3,-2);
\draw (0,2.75) node {7};
%\draw (3,0) node{3};
\end{tikzpicture}
\end{center}

 Therefore $d(v_5) \leq 5$, $d(v_2) \leq 4$, and $d(v_4) \leq 4$. We have $d(v_1) = 6$, and $d(v_3)$ may equal $6$. Thus, we choose either $S = \{v_1\}$ and $T = \{v_2\}$ or $S = \{v_1,v_3\}$ and $T = \{v_2,v_4\}$. 

It is clear that $S$ and $T$ satisfy conditions (1) and (3). We must check that $T$ satisfies condition (2).

We claim that it will suffice to show that neither $v_2$ nor $v_5$ is adjacent to a vertex $u$ such that $u$ is not on $C$ and $d(u) \geq 6$. Indeed, in both pairings, the only vertices which can appear in $T$ are $v_2$ and $v_4$. Moreover, $v_4$ is only included in $T$ when $d(v_3) = 6$. Recall that $v_1v_3$ is a chord colored $6$. Observing the symmetry in this configuration,

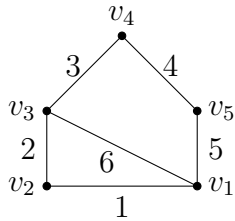
\begin{center}
\begin{tikzpicture}
\filldraw (-1,0) circle(0.05 cm);
\filldraw (1,0) circle(0.05 cm);
\filldraw (-1,1) circle(0.05 cm);
\filldraw (1,1) circle(0.05 cm);
\filldraw (0,2) circle(0.05 cm);
\draw (-1,0) -- (1,0) node[pos=0.5, below] {1};
\draw (-1,0) -- (-1,1) node[pos=0.5, left] {2};
\draw (1,0) -- (1,1) node[pos=0.5, right] {5};
\draw (-1,1) -- (0,2) node[pos=0.6, left] {3};
\draw (1,1) -- (0,2) node[pos=0.6, right] {4};
\draw (1,0) node[right] {$v_1$};
\draw (-1,0) node[left] {$v_2$};
\draw (-1,1) node[left] {$v_3$};
\draw (0,2) node[above] {$v_4$};
\draw (1,1) node[right] {$v_5$};
\draw (1,0) -- (-1,1) node[pos=0.6, below]{6};
\end{tikzpicture}
\end{center}
it is clear that if $v_5$ is not adjacent to a vertex not on $C$ of degree at least $6$, then the same argument implies that $v_4$ is not
adjacent to a vertex not on $C$ of degree at least $6$ when $d(v_3) = 6$.

Suppose first that $v_2$ is adjacent to a vertex $u$ not on $C$ with $d(u) \geq 6$. 
We have established already that $c(v_2u) = 4$. Furthermore, since $d(u) \geq 6$, $u$ has at least one neighbor, say $x$, not on $C$. It is easy to see that $c(xu)$ must be $6$. Therefore, $u$ has only one neighbor not on $C$, and so $u$ is adjacent to every vertex on $C$. In particular, $uv_4$ is an edge. This is pictured below.

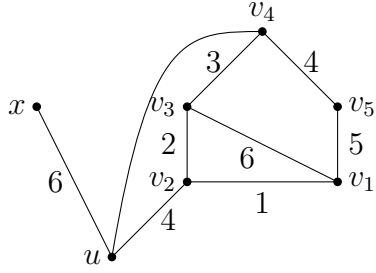
\begin{center}
\begin{tikzpicture}
\filldraw (-1,0) circle(0.05 cm);
\filldraw (1,0) circle(0.05 cm);
\filldraw (-1,1) circle(0.05 cm);
\filldraw (1,1) circle(0.05 cm);
\filldraw (0,2) circle(0.05 cm);
\draw (-1,0) -- (1,0) node[pos=0.5, below] {1};
\draw (-1,0) -- (-1,1) node[pos=0.5, left] {2};
\draw (1,0) -- (1,1) node[pos=0.5, right] {5};
\draw (-1,1) -- (0,2) node[pos=0.6, left] {3};
\draw (1,1) -- (0,2) node[pos=0.6, right] {4};
\draw (1,0) node[right] {$v_1$};
\draw (-1,0) node[left] {$v_2$};
\draw (-1,1) node[left] {$v_3$};
\draw (0,2) node[above] {$v_4$};
\draw (1,1) node[right] {$v_5$};
\draw (1,0) -- (-1,1) node[pos=0.6, below]{6};
\filldraw (-2,-1) circle(0.05 cm);
\draw (-2,-1) node[left] {$u$};
\draw (-1,0) -- (-2,-1) node[pos = 0.5, right] {4};
\filldraw (-3, 1) circle(0.05 cm);
\draw (-3,1) node[left] {$x$};
\draw (-2,-1) -- (-3,1) node[pos = 0.5, left] {6};
\draw (-2,-1) .. controls (-1.5,2) and (-1,2) .. (0,2);
\end{tikzpicture}
\end{center}

 The edge $uv_4$ is not in colored from $\{3,4,6\}$, so we must have $c(uv_4) = 5$, else $v_2uv_4v_3v_1v_5$ is rainbow. But if $c(uv_4) = 5$, then $xuv_4v_3v_2v_1$ is rainbow. We conclude that $v_2$ is not adjacent to a vertex of degree at least $6$ which is not on $C$.

Now suppose that $v_5$ is adjacent to a vertex $u$ such that $u$ is not on $C$ and $d(u) \geq 6$. To avoid a rainbow-$P_5$, we must have $c(v_5u) \in \{1,3\}$. Since $d(u) \geq 6$, $u$ has at least one neighbor, say $x$, which is not on $C$. Now note that if $c(v_5u) = 1$, then one of $xuv_5v_4v_3v_1$, $xuv_5v_1v_3v_2$, $xuv_5v_4v_3v_2$ is rainbow, regardless of $c(xu)$. So we may assume that $c(v_5u) = 3$. It can be checked that $c(xu)=2$ as otherwise we get a rainbow-$P_5$. This implies that $u$ is adjacent to every vertex of $C$. In particular, $u$ is adjacent to $v_2$, i.e., $v_2$ is adjacent to a vertex $u$ not on $C$ with $d(u) \geq 6$ which we have proved is a contradiction.

\smallskip

\textbf{Case 2.3:} $v_1$ is not adjacent to one of $v_3, v_4$.

\smallskip

Without loss of generality, $v_1v_4$ is not an edge. In order to achieve $d(v_1) = 6$, $v_1$ must be adjacent to $v_3$ and have three neighbors not on $C$, say $w_1, w_2, w_3$, with $c(v_1w_1) = 2$, $c(v_1w_2) = 3$, and $c(v_1w_3) = 4$.

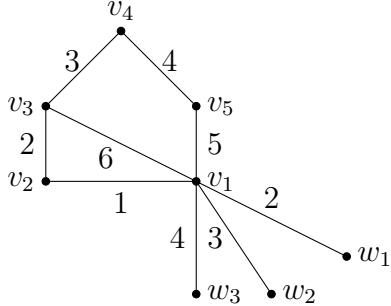
\begin{center}
\begin{tikzpicture}
\filldraw (-1,0) circle(0.05 cm);
\filldraw (1,0) circle(0.05 cm);
\filldraw (-1,1) circle(0.05 cm);
\filldraw (1,1) circle(0.05 cm);
\filldraw (0,2) circle(0.05 cm);
\draw (-1,0) -- (1,0) node[pos=0.5, below] {1};
\draw (-1,0) -- (-1,1) node[pos=0.5, left] {2};
\draw (1,0) -- (1,1) node[pos=0.5, right] {5};
\draw (-1,1) -- (0,2) node[pos=0.6, left] {3};
\draw (1,1) -- (0,2) node[pos=0.6, right] {4};
\draw (1,0) node[right] {$v_1$};
\draw (-1,0) node[left] {$v_2$};
\draw (-1,1) node[left] {$v_3$};
\draw (0,2) node[above] {$v_4$};
\draw (1,1) node[right] {$v_5$};
\draw (1,0) -- (-1,1) node[pos=0.6, below]{6};
\filldraw (3,-1) circle (0.05 cm);
\filldraw (2,-1.5) circle (0.05 cm);
\filldraw (1,-1.5) circle (0.05 cm);
\draw (3,-1) node[right] {$w_1$};
\draw (2,-1.5) node[right] {$w_2$};
\draw (1,-1.5) node[right]{$w_3$};
\draw (1,0) -- (3,-1) node[pos = 0.5, above] {2};
\draw (1,0) -- (2,-1.5) node[pos = 0.5, left] {3};
\draw (1,0) -- (1,-1.5) node[pos=0.5,left]{4};
\end{tikzpicture}
\end{center}

We first examine the degree of $v_2$. If $v_2u$ is an edge with $u$ not on $C$ (allowing $u = w_1, w_2$, or $w_3$), we must have $c(v_2u) = 4$ to avoid a rainbow-$P_5$. If $v_2v_4$ is an edge, then we must have $c(v_2v_4) = 6$, and if $v_2v_5$ is an edge, then we must have $c(v_2v_5) = 3$. So $d(v_2) \leq 5$. 

We next examine $v_4$. An edge $v_4u$ with $u$ not on $C$ cannot be colored $3$ or $4$, and must not be colored $1$ to avoid a rainbow-$P_5$. Moreover, if $c(v_4u) = 5$, then $u$ must equal $w_3$ in order to avoid a rainbow-$P_5$. Now, note that if $v_4w_3$ is an edge with $c(v_4w_3) = 5$, then the addition of the chord $v_2v_4$ (which must be colored $6$ by the above argument) creates a rainbow-$P_5$. These observations together imply $d(v_4) \leq 4$ ($v_4$ having either two neighbors on $C$ and at most two not on $C$, or three neighbors on $C$ and at most one not on $C$).  

Finally, we examine $v_5$. We have already seen that if $v_2v_5$ is present, then $c(v_2v_5) = 3$. If $v_3v_5$ is an edge, then we must have $c(v_3v_5)$ is $1$ or $7$. Finally, if $v_5u$ is an edge with $u$ not on $C$, then $c(v_5u) \neq 2$. Thus, any edge incident to $v_5$ must be colored from $\{1,3,4,5\}$, so $d(v_5) \leq 4$.

It is possible that $d(v_3) = 6$. So we choose either $S = \{v_1\}$ and $T = \{v_5\}$ or $S = \{ v_1, v_3\}$ and $T = \{v_4, v_5\}$. By an argument analogous to that in Case $2.2$, this produces an $S,T$ pair.
\renewcommand{\qedsymbol}{$\blacksquare$}
\end{proof}

With this claim established, we are now prepared to finish the proof. For each rainbow-$C_5$ in $G$ containing a vertex of degree greater than $5$, select an $S,T$ pair. Let $V'$  be the set of vertices of $G$ contained in at least one rainbow-$C_5$ and $U$ the set of vertices which are placed in $S,T$ pairs. Then the average degree of the vertices in $V'$ is
\[
\frac{\underset{v \in V'}{\sum} d(v)}{|V'|} = \frac{\underset{v \in U}{\sum} d(v) + \underset{v \in V' \setminus U}{\sum} d(v)}{|U| + |V' \setminus U|} \leq 5
\]
as
the conditions satisfied by $S,T$ pairs are sufficient to imply that $\underset{v \in U}{\sum} d(v) \leq 5|U|$, and we must have $\underset{v \in V' \setminus U}{\sum} d(v) \leq 5|V' \setminus U|$ since every vertex of degree greater than $5$ in $V'$ is in $U$.
\end{proof}

\begin{remark}
    As in the case of $P_3$ and $P_4$, the proof of Theorem~\ref{length-5} can be adapted to show that the only rainbow-$P_5$-free graphs attaining $\ex^*(n,C_5,P_5)$ are $5$-regular. We exclude the details as they involve further analysis in the subcases in the proof.
\end{remark}


\begin{thebibliography}{99}

\bibitem{AS} N. Alon, C. Shikhelman. Many $T$ copies in $H$-free graphs. Journal of Combinatorial Theory, Series B \textbf{121}, 146--172, 2016.

\bibitem{DLS13} S. Das, C. Lee, B. Sudakov. Rainbow Tur\'an problem for even cycles. European Journal of Combinatorics \textbf{34}, 905--915, 2013.

\bibitem{EGyM19} B. Ergemlidze, E. Gy\H{o}ri, A. Methuku. On the Rainbow Tur\'an number of paths. The Electronic Journal of Combinatorics \textbf{26}(1), P1.17, 2019.

\bibitem{GMMP} D. Gerbner, T. M\'esz\'aros, A. Methuku, C. Palmer. Generalized rainbow Tur\'an problems.
arXiv:1911.06642, 2019

\bibitem{GJ} W.T. Gowers, B. Janzer. Generalizations of the Ruzsa-Szemer\'edi and rainbow Tur\'an problems for cliques. arXiv:2003.02754, 2020

\bibitem{Ha} A. Halfpap. Maximal rainbow-$P_4$-free graphs are $4$-regular. (unpublished), 2019

\bibitem{jps}  D. Johnston, C. Palmer, A. Sarkar. Rainbow Tur\'an Problems for Paths and Forests of Stars. The Electronic Journal of Combinatorics \textbf{24}(1), P1--34, 2017.
%      
\bibitem{jr} D. Johnston, P. Rombach. Lower bounds for rainbow Tur\'{a}n numbers of paths and other trees. arXiv:1901.03308, 2019
%
\bibitem{kmsv}   P. Keevash, D. Mubayi, B. Sudakov, J. Verstra{\"e}te. Rainbow Tur\'an problems. Combinatorics, Probability and Computing \textbf{16}(1), 109--126, 2007.

\end{thebibliography}
\end{document}